\documentclass[12pt,reqno]{article}

\usepackage[usenames]{color}
\usepackage{amssymb}
\usepackage{graphicx}
\usepackage{amscd}

\usepackage{amsthm}
\newtheorem{theorem}{Theorem}

\newtheorem{proposition}[theorem]{Proposition}
\newtheorem{corollary}[theorem]{Corollary}

\theoremstyle{definition}

\newtheorem{example}[theorem]{Example}

\usepackage[colorlinks=true,
linkcolor=webgreen, filecolor=webbrown,
citecolor=webgreen]{hyperref}

\definecolor{webgreen}{rgb}{0,.5,0}
\definecolor{webbrown}{rgb}{.6,0,0}

\usepackage{color}

\usepackage{float}

\usepackage{graphics,amsmath,amssymb}
\usepackage{amsfonts}
\usepackage{latexsym}
\usepackage{epsf}

\newcommand{\seqnum}[1]{\href{http://oeis.org/#1}{\underline{#1}}}

\begin{document}

\begin{center}
\vskip 1cm{\LARGE\bf On a transformation of Riordan moment sequences} \vskip 1cm \large
Paul Barry\\
School of Science\\
Waterford Institute of Technology\\
Ireland\\
\href{mailto:pbarry@wit.ie}{\tt pbarry@wit.ie}
\end{center}
\vskip .2 in

\begin{abstract} We define a transformation that associates certain exponential moment sequences with ordinary moment sequences in a natural way. The ingredients of this transformation are series reversion, the Sumudu transform (a variant of the Laplace transform), and the inverting of generating functions. This transformation also has a simple interpretation in terms of continued fractions. It associates lattice path objects with permutation objects, and in particular it associates the Narayana triangle with the Eulerian triangle. \end{abstract}

\section{Introduction}

In this note we study relationships between moment sequences that are defined by ordinary Riordan arrays \cite{Book, SGWW, Spru} and by certain exponential Riordan arrays \cite{Book, Survey}, which may be described as Eulerian. The theory of orthogonal polynomials \cite{Chihara, Gautschi, Chebyshev, Szego} defined by Riordan arrays has been much studied \cite{Classical, CB, Book, Barry_moments, Meixner}. For ordinary Riordan arrays, the associated orthogonal polynomials are generalized Chebyshev polynomials. In this note, we shall define a mapping $\mathcal{T}$, whose inverse $\mathcal{T}^{-1}$ maps a subset of the set of moments defined by ordinary Riordan arrays to a set of moments whose exponential generating function is ``Eulerian''. A central role in this mapping is played by the Sumudu transform \cite{Belgacem2, Belgacem1, Watugala}.

Many sequences and triangles in this note are referenced by their $Annnnnn$ number in the On-Line Encyclopedia of Integer Sequences \cite{SL1, SL2}, an invaluable tool for notes such as this one. 

All matrices in this note are of infinite extent; we show in each case a relevant truncation. The operator $[x^n]$ is the operator that extracts the coefficient of $x^n$ in a power series \cite{MC}. Ordinary Riordan arrays are designated by $(g,f)$ while exponential Riordan arrays are designated by $[g,f]$.

\section{Motivation}
We begin with the Eulerian generating function
$$ E(x,y; a,b)= \frac{(1-y) e^{a x(1-y)}}{1-y e^{bx(1-y)}}.$$

When $a=0$ and $b=1$, we get
$$E(x,y; 0,1)= \frac{1-y}{1-y e^{x(1-y)}},$$ which is the generating function of the Eulerian polynomials
$$1, y, y(y + 1), y(y^2 + 4y + 1), y(y^3 + 11y^2 + 11y + 1), y(y^4 + 26y^3 + 66y^2 + 26y + 1), \ldots,$$ with
coefficient array the Eulerian triangle \seqnum{A123125} that begins
$$\left(
\begin{array}{cccccc}
 1 & 0 & 0 & 0 & 0 & 0 \\
 0 & 1 & 0 & 0 & 0 & 0 \\
 0 & 1 & 1 & 0 & 0 & 0 \\
 0 & 1 & 4 & 1 & 0 & 0 \\
 0 & 1 & 11 & 11 & 1 & 0 \\
 0 & 1 & 26 & 66 & 26 & 1 \\
\end{array}
\right).$$

When $a=1$ and $b=1$, we get
$$ E(x,y; 1,1)= \frac{(1-y) e^{ x(1-y)}}{1-y e^{x(1-y)}},$$ which is the generating function of the variant Eulerian polynomials $S_n(y)$ that begin
$$1, 1, y, y(y + 1), y(y^2 + 4y + 1), y(y^3 + 11y^2 + 11y + 1), y(y^4 + 26y^3 + 66y^2 + 26y + 1), \ldots.$$ 
Recall that we have \cite{Petersen}
$$S_n(t)=\sum_{w \in S_n} t^{des(w)}.$$ 
These polynomials have a
coefficient array \seqnum{A173018} that begins
$$\left(
\begin{array}{cccccc}
 1 & 0 & 0 & 0 & 0 & 0 \\
 1 & 0 & 0 & 0 & 0 & 0 \\
 1 & 1 & 0 & 0 & 0 & 0 \\
 1 & 4 & 1 & 0 & 0 & 0 \\
 1 & 11 & 11 & 1 & 0 & 0 \\
 1 & 26 & 66 & 26 & 1 & 0 \\
\end{array}
\right).$$

When $a=1$ and $b=2$ we get
$$ E(x,y; 1,2)= \frac{(1-y) e^{ x(1-y)}}{1-y e^{2x(1-y)}},$$ which is the generating function of the type $B$ Eulerian polynomials
$$1, y + 1, y^2 + 6y + 1, y^3 + 23y^2 + 23y + 1, y^4 + 76y^3 + 230y^2 + 76y + 1, \ldots,$$ with coefficient array \seqnum{A060187} that begins
$$\left(
\begin{array}{cccccc}
 1 & 0 & 0 & 0 & 0 & 0 \\
 1 & 1 & 0 & 0 & 0 & 0 \\
 1 & 6 & 1 & 0 & 0 & 0 \\
 1 & 23 & 23 & 1 & 0 & 0 \\
 1 & 76 & 230 & 76 & 1 & 0 \\
 1 & 237 & 1682 & 1682 & 237 & 1 \\
\end{array}
\right).$$

We wish now to associate with this exponential generating function (in $x$) an ordinary generating function. To this end we shall
\begin{itemize}
\item Invert the generating function $E(x,y;a,b)$ to get $1/E(x,y;a,b)$ 
\item Take the Sumudu transform of $1/E(x,y;a,b)$ to get an ordinary generating function $g(x,y;a,b)$ 
\item Use series reversion to form the generating function $G(x,y;a,b)=\frac{1}{x}\text{Rev}(x g(x,y,a,b))$. 
\end{itemize}

We shall designate this transformation pipeline by the symbol $\mathcal{T}$. Thus we have 
$$ \mathcal{T} (E(x,y;a,b))=G(x,y;a,b).$$ 
We now work out the form of $G(x,y;a,b)$. We have 
$$\frac{1}{E(x,y;a,b)} = \frac{1-y e^{bx(1-y)}}{(1-y) e^{a x(1-y)}}.$$ 
Now the Sumudu transform is given by the variant of the Laplace transform 
$$ \mathcal{S}(E)(x)=\frac{1}{x} \int_0^{\infty} E(t,y;a,b)e^{-t/x}\,dt.$$ 
We find that 
$$ \mathcal{S}\left(\frac{1}{E(t,y;a,b)}\right)(x)=\frac{1-x(a(y-1)+b)}{1+(1-y)(2a-b)x+a(a-b)(y-1)^2 x^2}.$$ 
To finish, we multiply this generating function by $x$, we revert the result, and then we divide this reversion by $x$. We obtain that 
$$G(x,y;a,b)=\frac{1+x(y-1)(2a-b)-\sqrt{1-2bx(y+1)+b^2 x^2(y-1)^2}}{2(ax(a-b)(y-1)^2+a(y-1)+b)}.$$ 
Writing $c(x)=\frac{1-\sqrt{1-4x}}{2x}$ for the generating function of the Catalan numbers \cite{Stanley} $C_n=\frac{1}{n+1}\binom{2n}{n}$, we can write this as 
$$G(x,y;a,b)=\frac{1}{1-x(b-2a)(y-1)}c\left(\frac{x(ax(a-b)(y-1)^2+ay-a+b)}{(1-x(b-2a)(y-1))^2}\right).$$ 
Thus 
$$\mathcal{T}\left(\frac{(1-y) e^{a t(1-y)}}{1-y e^{bt(1-y)}}\right)(x)=\frac{1}{1-x(b-2a)(y-1)}c\left(\frac{x(ax(a-b)(y-1)^2+a(y-1)+b)}{(1-x(b-2a)(y-1))^2}\right).$$
\begin{example} We take the case $a=0$, $b=1$. We find that 
$$\mathcal{T}\left(E(t,y;0,1)\right)(x)=\frac{1}{1+x(1-y)}c\left(\frac{x}{1+x(1-y)}\right).$$ 
This is the generating function for the Narayana polynomials that begin
$$ 1, y, y^2 + y, y^3 + 3y^2 + y, y^4 + 6y^3 + 6y^2 + y,\ldots$$ with coefficient array the Narayana triangle \seqnum{A090181} that begins
$$\left(
\begin{array}{cccccc}
 1 & 0 & 0 & 0 & 0 & 0 \\
 0 & 1 & 0 & 0 & 0 & 0 \\
 0 & 1 & 1 & 0 & 0 & 0 \\
 0 & 1 & 3 & 1 & 0 & 0 \\
 0 & 1 & 6 & 6 & 1 & 0 \\
 0 & 1 & 10 & 20 & 10 & 1 \\
\end{array}
\right).$$
Thus the Eulerian triangle is transformed to the Narayana triangle by $\mathcal{T}$. 
\end{example}
\begin{example}
We consider the case $a=1$ and $b=1$. We find that 
$$\mathcal{T}\left(E(x,y;1,1)\right)(x)=\frac{1}{1+x(y-1)}c\left(\frac{xy}{(1+x(y-1))^2}\right).$$ 
This is the generating function for the Narayana polynomials $N_n(y)$ that begin
$$1, 1, y + 1, y^2 + 3y + 1, y^3 + 6y^2 + 6y + 1,\ldots,$$ with coefficient array the Narayana triangle \seqnum{A131198} that begins
$$\left(
\begin{array}{cccccc}
 1 & 0 & 0 & 0 & 0 & 0 \\
 1 & 0 & 0 & 0 & 0 & 0 \\
 1 & 1 & 0 & 0 & 0 & 0 \\
 1 & 3 & 1 & 0 & 0 & 0 \\
 1 & 6 & 6 & 1 & 0 & 0 \\
 1 & 10 & 20 & 10 & 1 & 0 \\
\end{array}
\right).$$ 
We have \cite{Petersen} $$N_n(t)=\sum_{w \in S_n(231)}t^{des(w)}.$$ 
\end{example}
\begin{example}
We consider the $B$ type case $a=1$ and $b=2$. We find that 
$$\mathcal{T}\left(E(t,y;1,2)\right)(x)=c\left(x(1+y-x(y-1)^2)\right).$$ This is the generating function of the polynomials that begin 
$$1, y + 1, y^2 + 6y + 1, y^3 + 19y^2 + 19y + 1, y^4 + 48y^3 + 126y^2 + 48y + 1,\ldots,$$ with coefficient array that begins
$$\left(
\begin{array}{ccccccc}
 1 & 0 & 0 & 0 & 0 & 0 & 0 \\
 1 & 1 & 0 & 0 & 0 & 0 & 0 \\
 1 & 6 & 1 & 0 & 0 & 0 & 0 \\
 1 & 19 & 19 & 1 & 0 & 0 & 0 \\
 1 & 48 & 126 & 48 & 1 & 0 & 0 \\
 1 & 109 & 562 & 562 & 109 & 1 & 0 \\
 1 & 234 & 2031 & 3916 & 2031 & 234 & 1 \\
\end{array}
\right).$$
This is thus the $\mathcal{T}$ transform of the type $B$ Eulerian triangle. Note that setting $y=1$ gives us the row sums, which are $2^n C_n=[x^n] c(2x)$. 
\end{example}
\begin{example}
If $a_n=n! [x^n] f(x)$ and $b_n=[x^n] \mathcal{T}(f(s))(x)$, then we shall write 
$$ b_n = \mathcal{T} a_n.$$ 
Thus we have, for instance, 
$$\mathcal{T} n! = C_n.$$ 
This follows because 
\begin{itemize}
\item We start with $f(x)=\frac{1}{1-x}$, which we invert to get $1-x$ 
\item The Sumudu transform of $1-t$ is $1-x$
\item We have $c(x)=\frac{1}{x} \text{Rev}(x(1-x))$. 
\end{itemize}
What follows is a short table of transform pairs.
\begin{center}
\begin{tabular} {|c|c|c|c|}
\hline OEIS & $a_n$ & $b_n=\mathcal{T} a_n$ & OEIS \\\hline
 \seqnum{A000142} & $n!$ & $C_n$ & \seqnum{A000108}\\\hline
 \seqnum{A049774} & Permutations without double falls & Motzkin numbers & \seqnum{A001006} \\\hline
 \seqnum{A097899} & Permutations with no runs of length $1$ & Motzkin sums & \seqnum{A005043} \\\hline
 \seqnum{A000670} & Fubini numbers & Little Schroeder numbers & \seqnum{A001003}\\\hline
 \seqnum{A001586} & Springer numbers & - & \seqnum{A052709} \\\hline
 \seqnum{A000629} & Cyclically ordered partitions & Large Schroeder numbers & \seqnum{A006318}\\
\hline 
\end{tabular}
\end{center}
\end{example}

We close this section by noting that for sequences with ordinary generating functions that can be put in the form $$G(x,y;a,b)=\frac{1}{1+x(b-2a)(1-y)}c\left(\frac{x(ax(a-b)(y-1)^2-a(1-y)+b)}{(1+x(b-2a)(1-y))^2}\right),$$ for appropriate choices of the parameters $a$, $b$ and $y$, we can define the inverse transform $\mathcal{T}^{-1}$. This is formed by reversing the above steps. 
\begin{itemize} 
\item Calculate $g(x)=\frac{1}{x} \text{Rev}(xG)$
\item Calculate the inverse Sumudu transform of $g(s)$ to get $e(t)$
\item Form $E(t)=\frac{1}{e(t)}$. 
\end{itemize}

\section{Riordan arrays and Riordan moment sequences}
All the sequences encountered thus far are examples of Riordan moment sequences. By this we mean that they are moment sequences associated with families of orthogonal polynomials that are defined either by ordinary or by exponential Riordan arrays. We recall that an ordinary Riordan array is an invertible lower-triangular matrix $R$ defined by two power series
$$g(x)=1+g_1 x + g_2 x^2 + \cdots,$$ and 
$$f(x)= f_1 x + f_2 x^2 + \cdots,$$ whose $(n,k)$ element $R_{n,k}$ is given by 
$$R_{n,k}=[x^n] g(x)f(x)^k.$$ We write $R=(g, f)$ to signify this.
An ordinary Riordan array $R$ defines a family of orthogonal polynomials if we have 
$$R = \left(\frac{1+ \lambda x + \mu x^2}{1+\alpha x + \beta x^2}, \frac{x}{1+\alpha x + \beta x^2}\right).$$ 
The inverse of the Riordan array $R=(g, f)$ is given by 
$$R^{-1}=\left(\frac{1}{g(\bar{f})}, \bar{f}\right),$$ where 
$\bar{f}(x)=\text{Rev}(f)(x)$ is the reversion of the power series $f(x)$, defined as the power series that satisfies $f(\bar{f}(x))=x$ and $\bar{f}(f(x))=x$. For 
$R = \left(\frac{1+ \lambda x + \mu x^2}{1+\alpha x + \beta x^2}, \frac{x}{1+\alpha x + \beta x^2}\right)$ we have 
$$R^{-1}=\left(\mu(x),\frac{1-\alpha x -\sqrt{1-2 \alpha x + x^2(\alpha^2-4\beta)}}{2 \beta x}\right),$$
where 
$$\mu(x)=\frac{2 \beta}{(\beta-\mu)\sqrt{x^2(\alpha^2-4 \beta)-2\alpha x+1}+x(2\beta \lambda-\alpha(\beta+\mu))+\beta+\mu}.$$ 
The sequence with generating function $\mu(x)$ then appears as the first column of the inverse matrix $R^{-1}$. This is the (ordinary) Riordan moment sequence associated with the family of orthogonal polynomials defined by $R$. We have that
$$\mu(x)=
\cfrac{1}{1-(\alpha-\lambda)x-
\cfrac{(\beta-\mu)x^2}{1-\alpha x-
\cfrac{\beta x^2}{1-\alpha x-
\cfrac{\beta x^2}{1-\cdots}}}}.$$ 
In an obvious notation, we write this as 
$$\mu(x)=\mathcal{J}(\alpha-\lambda,\alpha,\alpha,\ldots;\beta-\mu,\beta, \beta,\ldots),$$ 
where $\mathcal{J}$ stands for ``Jacobi''.  This thus is the form of the generating function of an ordinary Riordan moment sequence. The corresponding family of orthogonal polynomials $P_n(x)$ satisfies the associated three-term recurrence
$$P_n(x)=(x-\alpha) P_{n-1}(x)-\beta P_{n-2}(x),$$ with 
$P_0(x)=1$, $P_1(x)=x-\alpha+\lambda$, and $P_2(x)=x^2 + x (\lambda- 2 \alpha) + \alpha^2 - \alpha \lambda - \beta + \mu$. 
In the case that $\mu=0$, we obtain 
$$\mu(x)=\frac{1}{1-x(\alpha-2 \lambda)}c\left(\frac{x(\lambda-x(\alpha \lambda-\beta-\lambda^2))}{(1-x(\alpha-2 \lambda))^2}\right).$$ 
By the fundamental theorem of Riordan arrays, this can be written 
$$\mu(x)=\left(\frac{1}{1-x(\alpha-2 \lambda)},\frac{x(\lambda-x(\alpha \lambda-\beta-\lambda^2))}{(1-x(\alpha-2 \lambda))^2}\right)\cdot c(x).$$ 
\begin{proposition}
We have 
$$G(x,y;a,b)=\mathcal{J}(y(b-a)+a, b(y+1),b(y+1),\ldots; b^2y,b^2y,\ldots).$$ 
\end{proposition}
We now turn our attention to exponential Riordan arrays. An exponential Riordan array $R$ is an invertible lower-triangular matrix defined by two power series
$$g(x)=1+ g_1 \frac{x}{1!} + g_2 \frac{x^2}{2!} + g_3 \frac{x^3}{3!} + \cdots,$$ and 
$$f(x)= f_1 \frac{x}{1!} + f_2 \frac{x^2}{2!} + f_3 \frac{x^3}{3!} + \cdots.$$ The general $(n,k)$-th element of $R$ is then defined to be 
$$R_{n,k}=\frac{n!}{k!} [x^n] g(x)f(x)^k.$$ 
We write $[g, f]$ to denote this matrix. With every exponential Riordan array $R$ we can associate its production matrix $P_R$ which is defined to be the matrix $P_R=R^{-1} \overline{R}$, where $\overline{R}$ is the matrix $R$ with its first row removed \cite{ProdMat_0, ProdMat}. The matrix $R^{-1}$ will then be the coefficient array of a family of orthogonal polynomials if $P_R$ has a bivariate generating function of the form
$$e^{xy}(\alpha+ \beta x + y(1+\gamma x+ \delta x^2))=e^{xy}(Z(x)+ y A(x)),$$ where
$$A(x)=f'(\bar{f}(x)), \quad Z(x)=\frac{g'(\bar{f}(x))}{g(\bar{f}(x))}.$$ The matrix $P_R$ then begins
$$\left(
\begin{array}{ccccccc}
 \alpha  & 1 & 0 & 0 & 0 & 0 & 0 \\
 \beta  & \alpha +\gamma  & 1 & 0 & 0 & 0 & 0 \\
 0 & 2 \beta +2 \delta  & \alpha +2 \gamma  & 1 & 0 & 0 & 0 \\
 0 & 0 & 3 (\beta +2 \delta ) & \alpha +3 \gamma  & 1 & 0 & 0 \\
 0 & 0 & 0 & 4 (\beta +3 \delta ) & \alpha +4 \gamma  & 1 & 0 \\
 0 & 0 & 0 & 0 & 5 (\beta +4 \delta ) & \alpha +5 \gamma  & 1 \\
 0 & 0 & 0 & 0 & 0 & 6 (\beta +5 \delta ) & \alpha +6 \gamma  \\
\end{array}
\right).$$
Note that the diagonal terms form an arithmetic sequence, while the sub-diagonal terms, when divided successively by $1,2,3,\ldots$, also form an arithmetic sequence. 

The corresponding exponential Riordan moment sequence (given by the first column elements of $R$) will then have a generating function given by the continued fraction \cite{CFT, Viennot, Wall}
$$\mu_e(x)=
\cfrac{1}{1-\alpha x-
\cfrac{\beta x^2}{1-(\alpha+\gamma)x-
\cfrac{(2 \beta+2 \delta)x^2}{1-(\alpha+2 \gamma)x-
\cfrac{(3 \beta+ 6 \delta)x^2}{1-\cdots}}}},$$ with coefficients drawn from the production matrix. 
Thus exponential Riordan moment sequences have ordinary generating functions given by 
$$\mathcal{J}(\alpha,\alpha+\gamma,\alpha+2 \gamma,\ldots;\beta,2(\beta+\delta),3(\beta+2 \delta),\ldots).$$ 
\begin{example} 
The exponential Riordan array \seqnum{A021009}
$$L=\left[\frac{1}{1+x}, \frac{x}{1+x}\right]$$ with general term 
$$L_{n,k}=\frac{n!}{k!}(-1)^{n-k} \binom{n}{k}$$ is the coefficient matrix of the scaled Laguerre polynomials. 
The corresponding moments are $n!$, given by the initial column of the inverse array
$$ \left[\frac{1}{1-x}, \frac{x}{1-x}\right].$$ The production matrix of this array begins 
$$\left(
\begin{array}{cccccc}
 1 & 1 & 0 & 0 & 0 & 0 \\
 1 & 3 & 1 & 0 & 0 & 0 \\
 0 & 4 & 5 & 1 & 0 & 0 \\
 0 & 0 & 9 & 7 & 1 & 0 \\
 0 & 0 & 0 & 16 & 9 & 1 \\
 0 & 0 & 0 & 0 & 25 & 11 \\
\end{array}
\right).$$ Thus $n!$, which has exponential generating function $\frac{1}{1-x}$, has an ordinary generating function given by 
$$\cfrac{1}{1-x-
\cfrac{x^2}{1-3x-
\cfrac{4x^2}{1-5x-
\cfrac{9x^2}{1-7x-\cdots}}}}.$$ 
\end{example}
With regard to the exponential generating function $E(x,y;a,b)$ (in $x$), we have the following result.
\begin{theorem} The exponential Riordan array 
$$\left[\frac{1+bxy}{(1+bx)^{a/b}}, \frac{1}{b(1-y)}\ln\left(\frac{1+bx}{1+bxy}\right)\right]$$ is the coefficient array of the family of orthogonal polynomials whose moment sequence has exponential generating function $E(x,y;a,b)$.
\end{theorem}
\begin{proof} The inverse of the above Riordan array is given by 
$$\left[\frac{(1-y) e^{a x(1-y)}}{1-y e^{bx(1-y)}}, \frac{e^{bx}-e^{bxy}}{b(e^{bxy}-ye^{bx})}\right].$$ 
We find that $$A(x)=(1+bx)(1+bxy) \quad\text{and}\quad Z(x)=a(1-y)+by+b^2y x.$$ 
This allows us to calculate the  production matrix of the inverse Riordan array, which begins
$$\left(
\begin{array}{ccccc}
 y (b-a)+a & 1 & 0 & 0 & 0 \\
 b^2 y & y (2 b-a)+a+b & 1 & 0 & 0 \\
 0 & 4 b^2 y & y (3 b-a)+a+2 b & 1 & 0 \\
 0 & 0 & 9 b^2 y & y (4 b-a)+a+3 b & 1 \\
 0 & 0 & 0 & 16 b^2 y & y (5 b-a)+a+4 b \\
\end{array}
\right).$$ 
 This indicates that $E(x,y;a,b)$ is the exponential generating function of an exponential Riordan moment sequence which has an ordinary generating function given by
$$\cfrac{1}{1-(y(b-a)+a)x-
\cfrac{b^2 yx^2}{1-(y(2b-a)+a+b)x-
\cfrac{4b^2 yx^2}{1-(y(3b-a)+a+2b)x-
\cfrac{9b^2 yx^2}{1-\cdots}}}}.$$ 
\end{proof}
\begin{corollary}
The ordinary generating function of the exponential Riordan moment sequence with e.g.f. $E(x,y;a,b)$ is given by $$\mathcal{J}(y(b-a)+a,y(b-a)+a+b(1+y),y(b-a)+a+2b(1+y),\ldots;b^2y, 4b^2y, 9b^2y,\ldots).$$ 
\end{corollary}
This is the Eulerian case, where $\beta=\delta=b^2 y$. 
We are now in a position to interpret the transformation $\mathcal{T}$ in terms of ordinary generating functions expressed as continued fractions. The description is easiest in the direction of $\mathcal{T}^{-1}$.

We have that $\mathcal{T}^{-1}$ maps the ordinary Riordan moments defined by the Riordan array
$$\left(\frac{1+(a(y-1)+b)x}{1+b(y+1)x+b^2yx^2},\frac{x}{1+b(y+1)x+b^2yx^2}\right)$$ to the exponential Riordan moments defined by the exponential Riordan array
$$\left[\frac{1+bxy}{(1+bx)^{a/b}}, \frac{1}{b(1-y)}\ln\left(\frac{1+bx}{1+bxy}\right)\right].$$ 
In both cases, the moment sequences are described by the first column elements of the respective inverse arrays. 
In terms of continued fractions, we have that $\mathcal{T}^{-1}$ maps 

$$\mathcal{J}(y(b-a)+a, b(y+1),b(y+1),\ldots; b^2y,b^2y,\ldots)$$ to 

$$\mathcal{J}(y(b-a)+a,y(b-a)+a+b(1+y),y(b-a)+a+2b(1+y),\ldots;b^2y, 4b^2y, 9b^2y,\ldots).$$

Thus the first elements are mapped to their partial sums, while the second terms are scaled by $(1,4,9,\ldots)$. Reversing this process (taking first differences; dividing by $(1,4,9,\ldots)$) now gives us the effect of $\mathcal{T}$. 

\section{Further results}
\begin{proposition} The moment sequence $\mu_n$ defined by the ordinary Riordan array 
$$\left(\frac{1-\alpha x}{1+\beta x+\gamma x^2}, \frac{x}{1+\beta x+\gamma x^2}\right)$$ satisfies 
$$ \mu_n = [x^n] \frac{1}{x} \text{Rev}\left(\frac{x(1+\alpha x)}{1+(\beta+2\alpha)x+(\alpha^2+\alpha \beta+\gamma)x^2}\right).$$
\end{proposition}
\begin{proof}
The moment sequence is the first column of the inverse array. This  has generating function 
$$\frac{2}{1-(2\alpha+\beta)x+\sqrt{1-2\beta x+(\beta^2-4\gamma)x^2}},$$ which is given by the reversion in the statement of the proposition.
\end{proof}
\begin{corollary} The generating function $G(x,y;a,b)=\mathcal{T}(E(t,y;a,b))(x)$ is the generating function of the moment sequence of the family of orthogonal polynomials defined by the ordinary Riordan array 
$$\left(\frac{1-(a(y-1)+b)x}{1+b(1+y)x+b^2yx^2}, \frac{x}{1+b(1+y)x+b^2yx^2}\right)$$
\end{corollary}
\begin{proof} We compare the coefficients $\alpha, \beta, \gamma$ in the proposition with the coefficients in 
$\mathcal{S}\left(\frac{1}{E(t,y;a,b)}\right)(x)$.  Thus 
$$\alpha=-a(y-1)+b, beta+2\alpha=(1-y)(2a-b), \alpha^2+\alpha \beta+\gamma=a(a-b)(y-1)^2.$$ We solve these equations for $\alpha, \beta, \gamma$. 
\end{proof}
\begin{proposition} Let $\mu_o(x)$ be the generating function of the moment sequence of the ordinary Riordan array 
$$\left(\frac{1-\alpha x}{1+\beta x+ \gamma x^2},\frac{x}{1+\beta x+ \gamma x^2}\right).$$ Then 
\begin{itemize}
\item $\mu_0(x)=\mathcal{J}(\beta+\alpha,\beta,\beta,\ldots;\gamma,\gamma,\gamma,\ldots)$.
\item The $\mathcal{T}^{-1}$ transform of $\mu_0(x)$ is 
$$\mu_e(x)=\frac{\sqrt{\beta^2-4 \gamma}e^{\frac{1}{2}(2\alpha+\beta)t}}
{\sqrt{\beta^2-4\gamma}\cosh\left(\frac{1}{2}\sqrt{\beta^2-4\gamma}t\right)-\beta \sinh\left(\frac{1}{2}\sqrt{\beta^2-4\gamma}t\right)}$$
\item The ordinary generating function of the transformed moment sequence is given by 
$$\mathcal{J}(\alpha+\beta, \alpha+2 \beta,\alpha+3 \beta, \ldots; \gamma,4\gamma, 9 \gamma,\ldots).$$
\end{itemize}
\end{proposition}
\section{Further examples}
\begin{example} The Motzkin numbers $M_n=\sum_{k=0}^{\lfloor \frac{n}{2} \rfloor} \binom{n}{2k}C_k$ are the moments of the Riordan array
$$\left(\frac{1}{1+x+x^2}, \frac{x}{1+x+x^2}\right),$$ where $a=0$, $b=1$, $c=1$. In this case we find that 
$$M_n = [x^n] \frac{1}{x} \text{Rev}\left(\frac{x}{1+x+x^2}\right).$$ 
The generating function of $M_n$ is thus $\frac{1-x-\sqrt{1-2x-3x^2}}{2x^2}$, which can be expressed as the continued fraction 
$$\cfrac{1}{1-x-
\cfrac{x^2}{1-x-
\cfrac{x^2}{1-x-\cdots}}},$$ or 
$$\mathcal{J}(1,1,1,\ldots; 1,1,1,\ldots).$$ 
Thus the $\mathcal{T}^{-1}$ transform of the Motzkin sequence has an ordinary generating function given by 
$$\mathcal{J}(1,2,3,\ldots; 1,4,9,\ldots).$$  This is \seqnum{A049774}, the number of permutations of $n$ elements not containing the consecutive pattern $123$. 
\end{example} 
\begin{example} 
The so-called ``Motzkin sums'' $MS_n=\sum_{k=0}^n \binom{n}{k}(\binom{k}{n-k}-\binom{k}{n-k-1})$ \seqnum{A005043} are the moments of the Riordan array
$$\left(\frac{1+x}{1+x+x^2}, \frac{x}{1+x+x^2}\right),$$ where $a=1$, $b=1$, $c=1$. In this case we find that
$$SM_n = [x^n] \frac{1}{x} \text{Rev}\left(\frac{x(1-x)}{1-x+x^2}\right).$$ In this case the generating function $\frac{1+x-\sqrt{1-2x-3x^2}}{2x(1+x)}$ is given by 
$$\mathcal{J}(0,1,1,\ldots; 1,1,1,\ldots).$$ Thus its $\mathcal{T}^{-1}$ transform is given by 
$$\mathcal{J}(0,1,2,3,\ldots; 1,4,9,\ldots).$$ This is \seqnum{A097899}, the number of permutations of $[n]$  with no runs of length $1$. 
\end{example}
\begin{example} We consider the sequence $a_n$  \seqnum{A052186} which begins
$$ 1, 0, 1, 3, 14, 77, 497, 3676, \ldots.$$ 
The ordinary generating function of this sequence is 
$$\mathcal{J}(0,3,5,7,9,\ldots;1,4,9,16,25,\ldots).$$ 
The $\mathcal{T}$ transform of this sequence will therefore have generating function 
$$\mathcal{J}(0,3,2,2,2,\ldots;1,1,1,1,1,\ldots).$$ 
Now the sequence with generating function 
$$\mathcal{J}(1,3,2,2,2,\ldots;1,1,1,1,1,\ldots)$$ which begins 
$$1, 1, 2, 6, 21, 78, 298, 1157, 4539, 17936, \ldots$$ is \seqnum{A129775}, the number of maximally clustered permutations in $S_n$ (those that avoid the patterns $3421$, $4312$ and $43214$). Thus the image of \seqnum{A052186} by $\mathcal{T}$ is the INVERT$(-1)$ transform of \seqnum{A129775}. 
Note that $a_n$ has the integral representation \cite{MK} 
$$a_n=\int_0^{\infty} \frac{x^n e^x}{(Ei(x)+e^x)^2+\pi^2}\,dx+\frac{-\mathfrak{p}}{\mathfrak{p}-1},$$ where 
$\mathfrak{p} \approx 0.434818$ is a simple pole of $\phi(z)=\frac{z}{1-e^z E_1(z)}-z$. 
\end{example}
\begin{example} We consider the $\mathcal{T}^{-1}$ transform of the sequence \seqnum{A064641}, which has a generating function 
$$\frac{1-x-\sqrt{1-6x-3x^2}}{2x(1+x)}=\frac{1}{x}\text{Rev}\left( \frac{x(1-x)}{1+x+x^2}\right).$$  
The general term of this sequence $b_n$ has the integral representation 
$$b_n = \frac{1}{2 \pi} \int_{3-2\sqrt{3}}^{3+2\sqrt{3}} x^n \frac{\sqrt{3(1+2x)-x^2}}{1+x}\,dx.$$ 
This sequence counts the number of paths from $(0,0)$ to $(n,n)$ not rising above $y=x$, using steps $(1,0)$, $(0,1)$, $(1,1)$ and $(2,1)$.

Letting $g(x)=\frac{1-x}{1+x+x^2}$, we calculate the inverse Laplace transform of $\frac{1}{s}f\left(\frac{1}{s}\right)$. We obtain 
$$e^{-t/2}\left(\cos\left(\frac {\sqrt{3}t}{2}\right)-\sqrt{3}\sin\left(\frac{\sqrt{3}t}{2}\right)\right).$$ 
Thus the desired transform is the exponential generating function 
$$ \frac{e^{t/2}}{\cos\left(\frac {\sqrt{3}t}{2}\right)-\sqrt{3}\sin\left(\frac{\sqrt{3}t}{2}\right)}.$$ 
This expands to the sequence that begins 
$$1, 2, 7, 35, 232, 1919, 19045, 220502, 2917663, 43431983,\ldots.$$ 
In this case we have  
$$\mathcal{J}(2,3,3,3,\ldots;3,3,3,3) \rightarrow \mathcal{J}(2,5,8,11,\ldots;3,12,27,48,\ldots).$$ 
\end{example}

\begin{example} We have seen that the $\mathcal{T}^{-1}$ transform operates on moment sequences defined by Riordan arrays of the form 
$$\left(\frac{1-ax}{1+bx+cx^2}, \frac{x}{1+bx+cx^2}\right).$$ 
However, ordinary Riordan arrays of the form 
$$\left(\frac{1-ax-bx^2}{1+cx+dx^2}, \frac{x}{1+cx+dx^2}\right)$$ also define moment sequences. To see the nature of the obstruction, we take the example of the matrix 
$$\left(\frac{1-x-x^2}{1+x+x^2}, \frac{x}{1+x+x^2}\right).$$ 
Then the relevant moment sequence has generating function 
$$G(x)=\frac{1}{\sqrt{1-2x-3x^2}-x}=\frac{\sqrt{1-2x-3x^2}-x}{1-2x-4x^2}.$$ This expands to give the sequence \seqnum{A111961}, which begins 
$$1, 2, 6, 18, 56, 176, 558, 1778, 5686, 18230, \ldots.$$
Its generating function is equal to 
$$\mathcal{J}(2,1,1,1,\ldots; 2,1,1,1,\ldots).$$ 
We now attempt to apply the $\mathcal{T}^{-1}$ transform in its ``reversion-Sumudu$^{-1}$-inverting'' form. 
We find that 
$$\frac{1}{x}\text{Rev}(xG(x))=\frac{\sqrt{1+2x+5x^2}-x}{1+2x+4x^2}.$$
Given its non-rational form, it is problematic to apply $\mathcal{S}^{-1}$ to this generating function. 
We can of course proceed as before with the continued fraction mapping, to get 
$$\mathcal{J}(2,1,1,1,\ldots; 2,1,1,1,\ldots) \rightarrow \mathcal{J}(2,3,4,5,\ldots;2,4,9,16,\ldots).$$
The image sequence then begins 
$$1, 2, 6, 22, 94, 454, 2454, 14766, 98678, 730422, \ldots.$$
\end{example}

\section{A note on the symmetric Eulerian triangle}
We have seen two variants of the Eulerian triangle that are associated with the sequence $1,4,9,\ldots$. 
The Pascal-like (or centrally symmetric) variant \seqnum{A008292} that begins 
$$\left(
\begin{array}{ccccccc}
 1 & 0 & 0 & 0 & 0 & 0 & 0 \\
 1 & 1 & 0 & 0 & 0 & 0 & 0 \\
 1 & 4 & 1 & 0 & 0 & 0 & 0 \\
 1 & 11 & 11 & 1 & 0 & 0 & 0 \\
 1 & 26 & 66 & 26 & 1 & 0 & 0 \\
 1 & 57 & 302 & 302 & 57 & 1 & 0 \\
 1 & 120 & 1191 & 2416 & 1191 & 120 & 1 \\
\end{array}
\right)$$ differs from the other two forms. 
To see this, we consider the bivariate generating function 
$$\frac{e^{x(1+y)}(1-y)^2}{(y e^x- e^{y x})^2}.$$ 
Again, we can show that this is the generating function of an exponential Riordan moment sequence. 
In this case, we find that the ordinary generating function is 
$$\mathcal{J}(y+1, 2(y+1), 3(y+1), \ldots; 2y, 6y, 12y, 20y,\ldots)$$ thus associating it with the sequence 
$1,3,6,\ldots$ rather than $1,4,9,\ldots$. 

If we now use an analogue of the $\mathcal{T}$ transform to associate with this an ordinary Riordan moment sequence, we get
$$\mathcal{J}(y+1,2(y+1),\ldots;2y,6y,12y,\ldots) \rightarrow \mathcal{J}(y+1,y+1,y+1,\ldots;2y,2y,2y,\ldots).$$ This last term gives us the triangle with bivariate generating function 
$$\frac{1}{1-x(1+y)}c\left(\frac{2x^2y}{(1-x(1+y))^2}\right).$$ 
This triangle begins 
$$\left(
\begin{array}{ccccccc}
 1 & 0 & 0 & 0 & 0 & 0 & 0 \\
 1 & 1 & 0 & 0 & 0 & 0 & 0 \\
 1 & 4 & 1 & 0 & 0 & 0 & 0 \\
 1 & 9 & 9 & 1 & 0 & 0 & 0 \\
 1 & 16 & 38 & 16 & 1 & 0 & 0 \\
 1 & 25 & 110 & 110 & 25 & 1 & 0 \\
 1 & 36 & 255 & 480 & 255 & 36 & 1 \\
\end{array}
\right).$$

\section{Conclusion}
In this note we have shown that ordinary Riordan moment sequences, defined by Riordan arrays of the form $\left(\frac{1-\alpha x}{1+\beta x+ \gamma x^2},\frac{x}{1+\beta x+ \gamma x^2}\right)$, can be mapped to exponential Riordan moments of Eulerian type. Some interesting sequence pairings have been exhibited between lattice path theoretic sequences and permutation-based sequences. There remains the challenge of putting this algebraic-analytic mapping into a full combinatorial context. As an example, we have seen that 
$$\mathcal{T}: \sum_{w \in S_n} t^{des(w)} \rightarrow \sum_{w \in S_n(231)}t^{des(w)}.$$

\bigskip
\hrule
\bigskip
\noindent 2010 {\it Mathematics Subject Classification}: Primary
15B36; Secondary 33C45, 11B83, 11C20, 05A15, 44A10.
\noindent \emph{Keywords:} Riordan array, orthogonal polynomial, Chebyshev polynomial, Catalan number, moment sequence, Sumudu transform.

\bigskip
\hrule
\bigskip
\noindent Concerned with sequences
\seqnum{A000108},
\seqnum{A000142},
\seqnum{A000629},
\seqnum{A000670},
\seqnum{A001003},
\seqnum{A001006},
\seqnum{A001586},
\seqnum{A005043},
\seqnum{A005043},
\seqnum{A006318},
\seqnum{A008292},
\seqnum{A064641},
\seqnum{A021009},
\seqnum{A049774},
\seqnum{A049774},
\seqnum{A052186},
\seqnum{A052709},
\seqnum{A060187},
\seqnum{A090181},
\seqnum{A097899},
\seqnum{A097899},
\seqnum{A111961},
\seqnum{A123125},
\seqnum{A129775},
\seqnum{A131198}, and
\seqnum{A173018}.

\end{document}